\documentclass[12pt]{article}

\usepackage{iftex}

\ifXeTeX

  \usepackage{fontspec}
  \defaultfontfeatures{Ligatures=TeX}

  \newfontfamily\greekfont{NewCM10-Regular.otf}[
    BoldFont       = NewCM10-Bold.otf,
    ItalicFont     = NewCM10-Italic.otf,
    BoldItalicFont = NewCM10-BoldItalic.otf,
    Script         = Greek,
    Language       = Greek,
    Scale          = MatchLowercase
  ]

  \usepackage{xeCJK}

  \defaultCJKfontfeatures{Scale=MatchLowercase}

  \setCJKsansfont{HaranoAjiGothic-Medium.otf}[
    BoldFont = HaranoAjiGothic-Bold.otf,
  ]

  \newCJKfontfamily\japanesefont{HaranoAjiMincho-Regular.otf}[
    BoldFont = HaranoAjiMincho-Bold.otf,
  ]
  \newCJKfontfamily\chinesefont{FandolSong-Regular.otf}[
    BoldFont = FandolSong-Bold.otf,
  ]

\fi

="[Tempora-Regular.otf]" at 12pt
="[Tempora-Italic.otf]" at 12pt
="[Tempora-Bold.otf]" at 12pt
="[Tempora-BoldItalic.otf]" at 12pt
="[Tempora-Regular.otf]:slant=0.2" at 12pt

\newfontfamily\LMSS   {lmsans10-regular.otf}
\newfontfamily\LMSSB  {lmsans10-bold.otf}
\newfontfamily\LMSSO  {lmsans10-oblique.otf}

\newfontfamily\LMSSDC {lmsansdemicond10-regular.otf}

\newfontfamily\LMSSQ  {lmsansquot8-regular.otf}
\newfontfamily\LMSSQO {lmsansquot8-oblique.otf}

\newcommand{\ssbf}      {\LMSSB\fontsize{12}{14.4}\selectfont}

\usepackage{amsmath}
\usepackage{amssymb}
\usepackage{amsthm}

\usepackage{etoolbox}
\AfterEndEnvironment{proof}{\par\addvspace{\topsep}}
\usepackage{amsfonts}
\usepackage{mathtools}
\usepackage[colorlinks=true,linkcolor=blue,citecolor=blue,urlcolor=blue]{hyperref}
\usepackage{enumitem}
\usepackage{booktabs}
\usepackage{geometry}
\usepackage{graphicx}
\usepackage{wrapfig}
\usepackage{caption}
\usepackage{subcaption}
\usepackage{tikz}
\usetikzlibrary{arrows.meta,calc}

\usepackage{titling}

\pretitle{\par\centering\ssbf}
\posttitle{\par\vskip 0.5em} 

\preauthor{\par\centering\normalsize}
\postauthor{\par\vskip 0em}   

\predate{\par\centering\normalsize}
\postdate{\par}

\newcommand{\affil}[2]{\textsuperscript{#1}#2}

\usepackage{abstract}

\usepackage{titlesec}
\titleformat{\section}
  {\normalfont\normalsize\bfseries} 
  {\thesection.}                   
  {0.5em}                          
  {}                               

\newtheoremstyle{scslnameonly}
  {1.5ex plus 0.2ex minus 0.1ex}
  {1.5ex plus 0.2ex minus 0.1ex}
  {\slshape}
  {}
  {\normalfont}
  {.}
  {0.5em}
  {\thmname{\scshape #1}\thmnumber{ #2}\thmnote{ (#3)}}
  
\newtheoremstyle{scnameonly}
  {1.5ex plus 0.2ex minus 0.1ex}
  {1.5ex plus 0.2ex minus 0.1ex}
  {\normalshape}
  {}
  {\normalfont}
  {.}
  {0.5em}
  {\thmname{\scshape #1}\thmnumber{ #2}\thmnote{ (#3)}}

\theoremstyle{scslnameonly}
\newtheorem{theorem}{Theorem}[section]
\newtheorem{proposition}[theorem]{Proposition}
\newtheorem{lemma}[theorem]{Lemma}

\theoremstyle{scnameonly}

\newcommand{\HH}{\mathbb{H}}
\newcommand{\MM}{\mathbb{M}}

\newcommand{\RR}{\mathbb{R}}
\newcommand{\SSS}{\mathbb{S}}

\newcommand{\dd}{\,\mathrm{d}}

\newcommand{\link}{\mathrel{\mathord{\prec}\mkern-3mu\mathord{\cdot}}}

\usepackage[hang,flushmargin]{footmisc}
\usepackage{xcolor}
\definecolor{BrightDarkGreen}{RGB}{0, 180, 0}
\definecolor{ElectricViolet}{RGB}{140, 20, 252} 

\usepackage{array}
\usepackage{tabularx}

\hypersetup{
  colorlinks=true,
  linkcolor=BrightDarkGreen,
  citecolor=ElectricViolet,
  urlcolor=blue
}

\setlist[itemize]{leftmargin=1.25em,itemsep=0.18em,topsep=0.2em}
\setlist[enumerate]{leftmargin=1.4em,itemsep=0.18em,topsep=0.2em}

\let\oldthebibliography\thebibliography
\renewcommand\thebibliography[1]{%
  \oldthebibliography{#1}%
  \setlength{\itemsep}{0pt}%
  \setlength{\parskip}{0pt}%
}

\title{The Boolean lattice $\mathcal{B}_6$ is not a sphere order}
\author{\normalsize%
  \textsl{David A. Meyer}%
  \\[0.3em] 
  {\normalfont\normalsize
    \affil{}{Department of Mathematics\\
             University of California, San Diego\\
             La Jolla, CA 92093-0112 USA\\
             \texttt{dmeyer@ucsd.edu}
            }%
  }%
}

\date{}

\begin{document}
\maketitle

\begin{abstract}
\noindent Using the equivalence of sphere order representations of 
posets with Minkowski space embeddings, we show that the Boolean 
lattice $\mathcal{B}_6$ is not a sphere order.  This is the first 
explicit finite example of a poset that is not a sphere order.
\end{abstract}


\section{Introduction}\label{sec:Introduction}

A partially ordered set (poset) is a set $P$ with a reflexive, 
antisymmetric, transitive relation $\preccurlyeq$.  For example, the
power set of an $n$ element set $[n]$ is a poset with the containment
relation:  for $A,B\in{\mathcal P}([n]) = [2]^n$,
$$
A\preccurlyeq B \quad\Longleftrightarrow\quad A\subseteq B.
$$
This poset is the \textsl{Boolean lattice} $\mathcal{B}_n$.  We write 
$A\prec B\in P$ if $A\preccurlyeq B$ and $A\ne B$; and we write 
$A\link B$ if $A\prec B$ and there is no element 
$C\in P$ such that $A\prec C\prec B$, \textit{i.e.}, $B$ 
\textsl{covers} $A$ (or $B$ is \textsl{linked} to $A$).

More geometrically, $d$ dimensional Minkowski space $\MM^d$ is the 
vector space $\RR\times\RR^{d-1}$ with the Lorentzian inner product
defined in (time, space) coordinates by
\begin{equation}\label{eq:Lorentzianinnerproduct}
  \langle (t,x),(s,y)\rangle_L = ts - x\cdot y,
\end{equation}
where $\cdot$ is the Euclidean inner product on $\RR^{d-1}$.  When 
$\langle Y-X,Y-X\rangle_L \ge 0$, $X = (t,x)$ and $Y = (s,y)$ are 
\textsl{causally related}; if in addition $t\le s$, then $X$ 
\textsl{causally precedes} $Y$ and we write 
$X\preccurlyeq Y\in\MM^d$.  Thus the points of Minkowski space are
partially ordered.  A natural question within the causal set approach
to a quantum theory of gravity~\cite{BLMS1987} is what is the smallest
dimension for which there is an order isomorphism (an 
\textsl{embedding}) $X : P\mapsto X(P)\subseteq\MM^d$; this is called 
the \textsl{Minkowski dimension} of a poset $P$~\cite{Meyer1989}.

For $d\ge 2$, $X\preccurlyeq Y\in\MM^d$ iff $s-t \ge \|y-x\|$, which 
is equivalent to the closed spherical ball with center $x$ and radius 
$t$ being contained in the closed spherical ball with center $y$ and 
radius $s$, provided $s,t\ge 0$.  Any embedding of a finite poset $P$ 
into $\MM^d$ can be time-translated so that all of the time 
coordinates are positive; thus it is equivalent to a realization of 
$P$ as a collection of closed spherical balls in $\RR^{d-1}$, 
partially ordered by containment, \textit{i.e.}, as a 
\textsl{sphere order}~\cite{BrightwellWinkler1989,Meyer1993}.

Felsner, Fishburn and Trotter have shown that not all finite posets 
are sphere orders, \textsl{in any dimension}~\cite{FFT1999}, but their
proof is Ramsey theoretic and does not exhibit any such poset, nor 
even give any indication of its cardinality (although they are 
``tempt[ed] to conjecture that there is a poset of modest size, say at 
most 100 points, which is not a sphere order'').  Explicit embeddings 
of Boolean lattices into Minkowski 
spaces~\cite{Meyer1989,Daughton1998}, however, have long suggested 
that $\mathcal{B}_6$ is not a sphere order, as conjectured by 
Brightwell and Winkler~\cite{BrightwellWinkler1989}.  Here we prove 
this conjecture, thereby showing that $\mathcal{B}_n$ is a sphere 
order iff $n\le 5$, and also that there is a 62-element poset, 
$\mathcal{B}_6^{\circ} 
 = \mathcal{B}_6^{\vphantom\circ}\setminus\{\emptyset,[6]\}$, 
that is not a sphere order.  El-Zahar and Fateen showed that a
realization of the latter as a sphere order would define a 
configuration of 62 points on $\SSS^4$ with antipodal atoms and 
complementary co-atoms and every maximal chain of total spherical 
length no more than $\pi$, whence they conjectured no such 
configuration exists~\cite{ElZaharFateen1998}.  We use the same 
Minkowski space geometry to prove this conjecture in every finite 
dimension.

\section{Conceptual framework}\label{sec:Conceptualframework}

In Lorentzian geometry the \textsl{closed future cone} is 
$\mathcal{C}^+ = \{(t,x)\in\MM^d\mid t\ge\|x\|\}$.  If there were an 
embedding of $\mathcal{B}_6$, the images of 
$A\subseteq B\subseteq [6]$, $X_A,X_B\in\MM^d$, would have to satisfy 
causal relations of the form $X_B - X_A\in\mathcal{C}^+$.  To show 
that there is no embedding, we need to show that these constraints are 
incompatible with the required non-relations.

Suppose a set of constraints is $\{L_k(X)\ge 0\}$.  Every nonnegative 
linear combination of satisfied inequalities is nonnegative.  A 
\textsl{certificate of infeasibility} instead combines the constraints 
to obtain an expression that must be strictly 
negative~\cite{BoydVandenberghe2004}.  In our case, the constraints 
are of the form that a vector does or does not belong to the cone 
$\mathcal{C}^+$, so the dual variables should be linear functionals 
that have single signs on that cone.  The Lorentzian inner 
product~\eqref{eq:Lorentzianinnerproduct} has this property:  if 
$W = (r,w)$ and $Y = (s,y)$ are in $\mathcal{C}^+$, then
\begin{equation}\label{eq:WYnonneg}
  \langle W,Y\rangle_L = rs - w\cdot y \ge rs - \|w\|\|y\| \ge 0.
\end{equation}
So the first idea of the proof is to choose a future directed causal 
vector $F_{AB}$ for each cover relation (link)
$A\link B$, with the property that 
$\langle F_{AB},X_B - X_A\rangle_L \ge 0$.

The natural setting for this is algebraic topology and discrete 
exterior calculus~\cite{DHLM2005}:  A finite poset $P$ is uniquely 
represented by its Hasse diagram, \textit{i.e.}, its transitive 
reduction to the cover (link) relation.  The links are oriented, so 
this is an oriented simplicial $1$-complex, with $0$-simplices 
$\{[A]\mid A\in P\}$, oriented $1$-simplices
$\{[AB]\mid A\link B\}$, and boundary map $\partial[AB] = [B] - [A]$.

Thus a map $X:P\to\MM^d$ is a vector-valued $0$-cochain, and its 
coboundary, or discrete exterior derivative, is the vector-valued 
$1$-cochain defined by
$$
(\text{d}X)_{AB} = X_B - X_A,
$$
the vectors which must be in $\mathcal{C}^+$ if $X$ is to be an
embedding.

A $1$-chain, on the other hand, is a linear combination 
$$
F = \sum_{[AB]} F_{AB}[AB],
$$
with boundary the $0$-chain defined by
\begin{equation}\label{eq:flux}
(\partial F)_A 
 =
\sum_{C\link A} F_{CA} - \sum_{A\link D} F_{AD}.
\end{equation}

If we take the $F_{AB}\in\MM^d$ we can define a pairing of chains
and cochains~\cite{Alexander1935} (analogous to differential forms and 
their integrals over chains in the smooth setting~\cite{deRham1931}) 
with the Lorentzian inner product~\eqref{eq:Lorentzianinnerproduct}.  
For a vector-valued $1$-cochain $\omega$, write
$$
\langle F,\omega\rangle
 =
\sum_{[AB]}\langle F_{AB},\omega_{AB}\rangle_L,
$$
and define the pairing of $0$-chains and $0$-cochains in the same way.  
In this setting Stokes' theorem is
\begin{equation}\label{eq:Stokestheorem}
  \langle F,\mathrm{d}X\rangle = \langle\partial F,X\rangle,
\end{equation}
so we have
\begin{equation}\label{eq:boundaryterms}
\sum_{A\link B}\langle F_{AB},X_B - X_A\rangle_L
=
\sum_A \langle(\partial F)_A,X_A\rangle_L.
\end{equation}
For any $A$ at which $(\partial F)_A = 0$, the terms which depend 
explicitly on $X_A$ in this sum cancel, so the second idea is to 
choose a \textsl{current} $F$ that is \textsl{conserved}, 
\textit{i.e.}, on internal (\textit{i.e.}, nonextremal) elements $A$, 
the total incoming and total outgoing currents in~\eqref{eq:flux} are 
equal.  

The third idea is to ensure that the non-relations force the remaining 
\textsl{boundary} terms in~\eqref{eq:boundaryterms} to be negative.  
The Lorentzian inner product~\eqref{eq:Lorentzianinnerproduct} also 
has the property that  
$$
X_B - X_A\not\in\mathcal{C}^+
\quad\Longleftrightarrow\quad
\exists V\in\mathcal{C}^+ : \langle V,X_B - X_A\rangle_L < 0.
$$
Thus we want to show that for any putative embedding $X$ we can choose 
$F$ so that
$$
0\le \langle F,\text{d}X\rangle = \langle\partial F,X\rangle
 =
\sum_{\alpha} c_{\alpha}\langle V_{\alpha},
                                X_{B_{\alpha}}-X_{A_{\alpha}}
                        \rangle_L
 < 0,
$$
for some $c_{\alpha} > 0$ and witnesses $V_{\alpha}$ for non-relations 
$(A_{\alpha},B_{\alpha})$, $\alpha\in[k]$, since this contradiction 
would constitute a certificate of infeasibility, as desired.

In the following we develop these three ideas to show that 
$\mathcal{B}_6$ cannot be embedded in any Minkowski space, and 
therefore is not a sphere order.

\section{Lorentzian geometry}\label{sec:Lorentziangeometry}

The conceptual framework described in the previous section depends 
upon phrasing the problem as the non-existence of an embedding of 
$\mathcal{B}_6$ into Minkowski space.  We begin the proof with some 
simple facts about Minkowski geometry.

\begin{lemma}\label{lem:cone}
The cone $\mathcal{C}^+$ is self-dual with respect to the Lorentzian
inner product:
$$
Y\in\mathcal{C}^+
 \quad\Longleftrightarrow\quad
\forall W\in\mathcal{C}^+ : \langle W,Y\rangle_L \ge 0.
$$
If $T$ is future-directed timelike and 
$Y\in\mathcal{C}^+\setminus\{0\}$, then $\langle T,Y\rangle_L > 0$.
Moreover, if $N = (u,n)$ is a future-directed \textrm{null} vector, 
\textit{i.e.}, a nonzero vector in $\mathcal{C}^+$ with $u = \|n\|$, 
and $W\in\mathcal{C}^+$, then
$$
\langle N,W\rangle_L = 0
 \quad\Longleftrightarrow\quad
\exists \lambda\ge 0 : W = \lambda N.
$$
Thus if a sum of future-directed causal vectors is a nonzero null
vector, every nonzero summand lies on the same null ray.
\end{lemma}

\begin{proof}
As we already observed in~\eqref{eq:WYnonneg}, if $Y\in\mathcal{C}^+$, 
then for all $W\in\mathcal{C}^+$, the Cauchy-Schwarz inequality 
implies $\langle W,Y\rangle_L \ge 0$.  Conversely, suppose 
$Y = (s,y)$.  If $y\ne 0$, then $W = (1,y/\|y\|)\in\mathcal{C}^+$, so 
$\langle W,Y\rangle_L = s - \|y\|\ge 0$, whence $Y\in\mathcal{C}^+$.
If $y = 0$, then set $W = (1,0)\in\mathcal{C}^+$, so 
$\langle W,Y\rangle_L = s\ge 0$, whence also $Y\in\mathcal{C}^+$.

If $T = (t,x)$ is future-directed timelike and 
$Y = (s,y)\in\mathcal{C}^+\setminus\{0\}$, then $t > \|x\|$ and 
$s\ge \|y\|$, with $s > 0$.  Hence
$$
\langle T,Y\rangle_L
 =
ts - x\cdot y
 \ge 
ts - \|x\|\|y\|
 \ge 
s(t - \|x\|)
 >
0.
$$

Let $W = (r,w)$.  Then
$$
\langle N,W\rangle_L = ur - n\cdot w \ge u(r - \|w\|)\ge 0,
$$
so equality forces $r = \|w\|$ and $w = (r/u)n$, whence 
$W = \lambda N$, with $\lambda = r/u\ge0$.

Finally, if $\sum W_i = N$ with $W_i\in\mathcal{C}^+$, then 
$\sum \langle N,W_i\rangle_L = 0$.  Thus each of the summands must be 
$0$ and each $W_i = \lambda_i N$ with $\lambda_i \ge 0$, so they all 
lie on the positive ray defined by $N$.
\end{proof}

We have introduced null vectors in Lemma~\ref{lem:cone} because the
witnesses we construct for non-relations will be null vectors, of the
form $V_{\alpha} = (1,v_{\alpha})$ with $\|v_{\alpha}\| = 1$.  The 
positive cone is invariant under orthochronous Lorentz 
transformations, \textit{i.e.}, elements of $\text{O}^+(1,d-1)$, which 
allows us to normalize sets of such witnesses:

\begin{lemma}\label{lem:center}
For $k\ge 2$ and $\alpha\in[k]$, let $V_{\alpha} = (1,v_{\alpha})$ 
where the $v_{\alpha}$ are distinct unit vectors in $\RR^{d-1}$.  
There exist an orthochronous Lorentz transformation $\Lambda$ and 
$\omega_{\alpha} > 0$ such that
$$
V'_{\alpha} = \frac{\Lambda V_{\alpha}}
                   {\omega_{\alpha}} 
            = (1,v'_{\alpha}),
$$
for distinct $v'_{\alpha}$ with $\sum_{\alpha}v'_{\alpha} = 0$.
\end{lemma}

\begin{proof}
First consider the case $k\ge 3$.  For 
$H\in\HH^{d-1} = \{H\in\mathcal{C}^+ \mid \langle H,H\rangle_L = 1\}$,
the hyperboloid of future-directed unit timelike vectors, let
$$
\Phi(H) = \sum_{\alpha=1}^k \log\langle H,V_{\alpha}\rangle_L,
$$
which is real since the Lorentzian inner products are all positive.
Write $H = (\cosh q,h\sinh q)$ with $h\in\RR^{d-1}$, $\|h\| = 1$, and
$q\ge 0$.  $\Phi(H)$ has a minimum at some finite value $q$ because 
$h$ can be the same as at most one of the distinct $v_{\alpha}$, and
as $q\to\infty$ all the other terms grow arbitrarily large.  More 
precisely,
$$
\langle H,V_{\alpha}\rangle_L 
 =
\cosh q - h\cdot v_{\alpha}\sinh q
 \ge
\cosh q - \sinh q
 =
e^{-q}
$$
lower bounds the inner product with \textsl{any} of the $V_{\alpha}$.
If $\max_{\alpha\ne\beta}v_{\alpha}\cdot v_{\beta} = \cos\theta < 1$,
then for any $h$, at least $k-1$ of the $v_{\alpha}$ satisfy 
$h\cdot v_{\alpha}\le\cos(\theta/2)$, so for the corresponding 
$V_{\alpha}$,
$$
\langle H,V_{\alpha}\rangle_L 
 =
\frac{e^q}{2}(1 - h\cdot v_{\alpha}) +
\frac{e^{-q}}{2}(1 + h\cdot v_{\alpha})
 \ge
\frac{e^q}{2}\bigl(1 - \cos(\theta/2)\bigr)
 =
Ce^q,
$$
with $C > 0$.  Thus
$$
\Phi(H) \ge \log(e^{-q}) + (k-1)\log (Ce^q)
        = (k-2)q + (k-1)\log C,
$$
which diverges as $q\to\infty$ for $k\ge3$.  Thus $\Phi$ has a 
minimum at some finite $H^*$.

At this minimum, the derivative of $\Phi$ in the direction 
$Z\in (H^*)^{\perp}$ vanishes:
$$
\big\langle Z,\sum_{\alpha=1}^k \frac{V_{\alpha}}
                                     {\langle H^*,V_{\alpha}\rangle_L}
\big\rangle_L = 0,
$$
so the second vector in the inner product must be a multiple of $H^*$.  
Taking the Lorentzian inner product with $H^*$ gives
\begin{equation}\label{eq:gradientiskH}
  \sum_{\alpha=1}^k \frac{V_{\alpha}}
                       {\langle H^*,V_{\alpha}\rangle_L}
  = kH^*.
\end{equation}

In the case $k = 2$, for any $H\in\HH^{d-1}$, let 
$a_{\alpha} = \langle H,V_{\alpha}\rangle_L$ so that 
$V_{\alpha} = a_{\alpha}H + b_{\alpha}$ with 
$b_{\alpha}\in H^{\perp}$.  Since $V_{\alpha}$ is null, 
$\langle b_{\alpha},b_{\alpha}\rangle_L = -a_{\alpha}^2$.  The 
bilinear form $-\langle\cdot,\cdot\rangle_L$ is a positive-definite
inner product on $H^{\perp}$, so the Cauchy-Schwarz inequality gives
$|\langle b_1,b_2\rangle_L|\le a_1a_2$.  Then
$$
0 < \langle V_1,V_2\rangle_L 
  = a_1a_2 + \langle b_1,b_2\rangle_L
  \le 2a_1a_2,
$$
where positivity follows from Lemma~\ref{lem:cone} since $V_1\ne V_2$.
Now,
$$
\Phi(H) = \log\langle H,V_1\rangle_L + \log\langle H,V_2\rangle_L
        = \log(a_1a_2)
        \ge \log(\langle V_1,V_2\rangle_L/2).
$$
This lower bound is achieved at 
$$
H^* = \frac{V_1 + V_2}
           {\sqrt{2\langle V_1,V_2\rangle_L}}
    \in\HH^{d-1},
$$
since $\langle H^*,V_{\alpha}\rangle_L = 
\sqrt{\langle V_1,V_2\rangle_L/2}$; furthermore, 
\eqref{eq:gradientiskH} holds with $k = 2$.

Now, for $k\ge2$, choose an orthochronous Lorentz transformation 
$\Lambda$ such that $\Lambda H^* = (1,0)$ and let 
$\omega_{\alpha} = \langle H^*,V_{\alpha}\rangle_L > 0$.  Since 
Lorentz transformations preserve the inner product, the time 
component of $\Lambda V_{\alpha}$ must be $\omega_{\alpha}$.  Thus 
$V'_{\alpha} = \Lambda V_{\alpha}/\omega_{\alpha} = (1,v'_{\alpha})$,
with $\|v'_{\alpha}\| = 1$, and the $v'_{\alpha}$ are distinct because
$\Lambda$ is invertible.  Finally, applying $\Lambda$ 
to~\eqref{eq:gradientiskH} gives
$$
\sum_{\alpha=1}^k V'_{\alpha} = k(1,0),
$$
so $\sum_{\alpha}v'_{\alpha} = 0$.
\end{proof}

We remark that $H^*$ is a conformal barycenter whose existence was
proved by Douady and Earle~\cite{DouadyEarle1986}.  This proof applies 
to a probability distribution comprised by a finite number of equal 
weight atoms, and is essentially a Lorentzian version of the argument 
of Cantarella and Schumacher~\cite{CantarellaSchumacher2022}.  The 
$k\ge 3$ case is \textsl{stable} in the terminology of Kapovich and 
Millson, while the $k = 2$ case is \textsl{nice semi-stable}~\cite{KapovichMillson1996}.

\section{Embedding constraints for $\mathcal{B}_n$}
\label{sec:EmbeddingconstraintsforBn}

As we described in Section~\ref{sec:Conceptualframework}, an embedding
$X : \mathcal{B}_n \to\MM^d$ must satisfy $X_B - X_A\in\mathcal{C}^+$ 
when $A\subseteq B\subseteq [n]$.  At the same time, if 
$A\not\subseteq B$ we will need a witness $V\in\mathcal{C}^+$ with 
$\langle V,X_B - X_A\rangle_L < 0$ to ensure that 
$X_B - X_A\not\in\mathcal{C}^+$.  The following lemma details the 
witness we will use to show that $\mathcal{B}_6$ is not embeddable 
into any Minkowski space.

\begin{lemma}\label{lem:witnesses}
Let $n\ge 2$.  If there is an embedding $X : \mathcal{B}_n\to\MM^d$, 
with $X_A = (t_A,x_A)$, then there exist distinct unit vectors 
$v'_i\in\RR^{d-1}$, $i\in[n]$ and an embedding 
$X' : \mathcal{B}_n\to\MM^d$ such that $\sum_i v'_i = 0$ and 
$\langle V'_i,X'_B - X'_A\rangle_L < 0$ whenever $i\in A\setminus B$, 
where $V'_i = (1,v'_i)$.
\end{lemma}

\begin{proof}
Let $A_i = \{i\}$ and $B_i = [n]\setminus\{i\}$.  Since $n\ge 2$,
$A_i$ and $B_i$ are incomparable, and since $X$ is an embedding, 
neither $X_{B_i} - X_{A_i}$ nor $X_{A_i} - X_{B_i}$ is in 
$\mathcal{C}^+$, so $x_{B_i}\ne x_{A_i}$.  Let
$$
v_i = \frac{x_{B_i} - x_{A_i}}
           {\|x_{B_i} - x_{A_i}\|}.
$$
Then
\begin{equation}\label{eq:witnessneg}
\langle V_i,X_{B_i} - X_{A_i}\rangle_L
 =
t_{B_i} - t_{A_i} - \|x_{B_i} - x_{A_i}\| 
 <
0.
\end{equation}
These witnesses are distinct:  if $V_i = V_j = V$ for $i\ne j$, then
the strict inequality~\eqref{eq:witnessneg} gives
$$
0 > \langle V,X_{B_i} - X_{A_i}\rangle_L + 
    \langle V,X_{B_j} - X_{A_j}\rangle_L
  = \langle V,X_{B_i} - X_{A_j}\rangle_L +
    \langle V,X_{B_j} - X_{A_i}\rangle_L
  \ge 0,
$$
a contradiction (the last inequality follows from $A_j\subseteq B_i$
and $A_i\subseteq B_j$, and Lemma~\ref{lem:cone}).

Now suppose $i\in A\setminus B$.  Then $A_i\subseteq A$ and 
$B\subseteq B_i$, so
$$
\langle V_i,X_B - X_A\rangle_L 
 =
\langle V_i,X_{B_i} - X_{A_i}\rangle_L -
\langle V_i,X_{B_i} - X_B\rangle_L -
\langle V_i,X_A - X_{A_i}\rangle_L
 <
0,
$$
since the first term is negative by~\eqref{eq:witnessneg} and the two
subtracted terms are nonnegative by Lemma~\ref{lem:cone}.

Finally, since the witnesses are distinct, Lemma~\ref{lem:center} 
applies, and the Lorentz transformation $\Lambda$ that makes the sum 
of the $v'_i$ vanish, composed with $X$, defines $X' = \Lambda X$,
which has the same causality relations.  In particular,
$$
\langle V'_i,X'_B - X'_A\rangle_L 
 = \frac{1}{\omega_i}\langle V_i,X_B - X_A\rangle_L
 < 0,
$$
since $\omega_i > 0$ by Lemma~\ref{lem:center}.
\end{proof}

\section{Local currents for $\mathcal{B}_n$ and their sums}
\label{sec:LocalcurrentsforBn}

Now drop the primes from the witnesses and the embedding we found in 
Lemma~\ref{lem:witnesses} and write $V_i = (1,v_i)$ with 
$v_i\in\RR^{d-1}$ and $i\in[n]$.  Using the notation from the proof of 
the lemma, by the witness inequalities and Stokes' 
theorem~(\ref{eq:Stokestheorem},\ref{eq:boundaryterms}), it would 
suffice to find a conserved future-directed causal current $F$ on the 
nontrivial layers of $\mathcal{B}_n$, with
\begin{equation}\label{eq:certificate}
  \partial F
  =
  \sum_{i=1}^n\lambda_i V_i\bigl([B_i]-[A_i]\bigr),
\end{equation}
with $\lambda_i > 0$.  In this case the total current leaving $A_i$ is 
null, so by Lemma~\ref{lem:cone} every outgoing current must lie on 
the null ray through $V_i$.  Thus we would have 
$F_{A_i,\{i,j\}} = a_{ij} V_i$, with $a_{ij}\ge 0$.  Then the current 
entering a pair $\{i,j\}$ is $a_{ij}V_i + a_{ji}V_j$.  When both 
coefficients are positive, the current is timelike and can split into 
causal currents with different directions as it leaves the pair.  The 
same null-ray rigidity holds for currents entering the co-atoms $B_i$.  
Removing these outer layers reduces the problem to one on the middle
layers, with boundary fluxes at pairs and co-pairs and conservation 
required at elements of the other layers.  Since we want to consider
at least all the middle layers in $\mathcal{B}_6$, the rest of this  
discussion motivating the choice of local currents presumes $n\ge 6$.

Complementation exchanges the prescribed sources and sinks.  Given a
current $F$, averaging $F_{AC}$ with $F_{C^{\text{c}}A^{\text{c}}}$ on 
each cover $A\link C$, where complements are taken in $[n]$, preserves 
the boundary and makes the current invariant under complementation.  
The flux supplied at each pair must then equal the flux removed at
its complementary co-pair.  A natural choice is equal coefficients
$a_{ij} = a_{ji} > 0$, so that this flux is a positive multiple of
$V_i + V_j$.

Next, each triple is linked to three pairs, so we can label the 
corresponding local currents $U_{ij,k}$, and assume they are symmetric 
in $i$ and $j$.  The same null rigidity hints at a possible form for 
these currents:  in the limiting case $v_i = v_j$ the flux leaving the 
pair would be null so each of the local currents $U_{ij,k}$ would have
to lie on the common null ray.  If we set 
\begin{equation}\label{eq:nmdef}
  n_{ij} = v_i\cdot v_j
  \quad\text{and}\quad
  m_{ij} = 1 - n_{ij} 
         = \langle V_i,V_j\rangle_L 
         = \frac12\|v_i - v_j\|^2
         \ge 0,
\end{equation}
this suggests that the coefficient of $v_k$ in $U_{ij,k}$ be 
proportional to $m_{ij}$ so that it would vanish were $v_i$ and $v_j$ 
equal.  Guessing that the coefficients of $v_i$ and $v_j$ are 
analogously proportional to $m_{jk}$ and $m_{ik}$, respectively, leads 
to the \textit{ansatz}:
\begin{equation}\label{eq:Uguess}
U_{ij,k}
 =
\bigl(\alpha m_{ij} + \beta(m_{ik}+m_{jk}),
      m_{jk}v_i + m_{ik}v_j + \gamma m_{ij}v_k
\bigr),
\end{equation}
where we have fixed an overall scale.  

If the current is to be conserved at triples, the sum of the three 
incoming currents must be equal to the sum of the outgoing currents.  
An easy way to simplify this condition is to make the spatial 
component of each sum vanish; it turns out that the remaining 
condition that the temporal components are equal can be met when 
$n = 6$.  The spatial part of $U_{ij,k}+U_{ik,j}+U_{jk,i}$ is
$$
(2 + \gamma)(m_{jk}v_i + m_{ik}v_j + m_{ij}v_k),
$$
so we take $\gamma = -2$ to make it $0$.

With this choice, the flux from the pair $\{i,j\}$ is
$$
\sum_{k\in[n]\setminus\{i,j\}} U_{ij,k}
 = 
\sum_{k\in[n]\setminus\{i,j\}} 
\bigl(\alpha m_{ij} + \beta(m_{ik}+m_{jk}),
      m_{jk}v_i + m_{ik}v_j - 2m_{ij}v_k
\bigr).
$$
To simplify this, recall that the $v_i$ sum to $0$, so
\begin{equation}\label{eq:summik}
  \sum_{k\in[n]\setminus\{i,j\}}m_{ik}
  =
  (n-2) - v_i\cdot\sum_{k\in[n]\setminus\{i,j\}}v_k
  =
  n-1 + v_i\cdot v_j
  =
  n - m_{ij}.
\end{equation}
Thus
\begin{equation}\label{eq:Uijksumalphabeta}
  \sum_{k\in[n]\setminus\{i,j\}} U_{ij,k}
  =
  \Bigl(2n\beta+\bigl((n-2)\alpha-2\beta\bigr)m_{ij},
        (n + m_{ij})(v_i + v_j)
  \Bigr).
\end{equation}
To make this sum a positive multiple of $V_i + V_j$, namely 
$(n + m_{ij})(V_i + V_j)$, we choose $\beta = 1$ and 
$\alpha = 4/(n-2)$.  In the case $n = 6$, this means $\alpha = 1$, so
that is how we will define $U_{ij,k}$.  It remains to show that these
local currents are future-directed and causal.  The following lemma 
establishes this for arbitrary triples of unit vectors.

\begin{lemma}\label{lem:triples}
For any three unit vectors $v_i,v_j,v_k\in\RR^{d-1}$, define
\begin{equation}\label{eq:Uijk}
  U_{ij,k} 
  = 
  (m_{ij} + m_{ik} + m_{jk},m_{jk}v_i + m_{ik}v_j - 2m_{ij}v_k).
\end{equation}
Then $U_{ij,k} = U_{ji,k}$ and $U_{ij,k}\in\mathcal{C}^+$.  In 
particular,
\begin{equation}\label{eq:Uijknorm2}
  \langle U_{ij,k},U_{ij,k}\rangle_L 
  =
  \frac32 m_{ij}\|v_i + v_j - 2n_{jk}v_k\|^2
  \ge
  0.
\end{equation}
\end{lemma}

\begin{proof}
That $U_{ij,k}$ is symmetric in $i$ and $j$ follows immediately from 
its definition~\eqref{eq:Uijk}; its time component is nonnegative
because the $v_i$ are unit vectors.  Write $a = m_{ik}$, $b = m_{jk}$, 
and $c = m_{ij}$.  Using $v_i\cdot v_k = 1 - a$, 
$v_j\cdot v_k = 1 - b$, and $v_i\cdot v_j = 1 - c$, we compute
\begin{align*}
  \langle U_{ij,k},U_{ij,k}\rangle_L
  &=
  (a+b+c)^2 - (bv_i + av_j - 2cv_k)\cdot(bv_i + av_j - 2cv_k)       \\
  &=
  (a+b+c)^2 - (a + b - 2c)^2 - 6abc                                 \\
  &=
  3c(2a + 2b - c - 2ab)                                             \\
  &=
  3m_{ij}(1 + n_{ij} - 2n_{ik}n_{jk}).
\end{align*}
Also,
$$
\|v_i + v_j - 2n_{jk}v_k\|^2 
 =
2(1 + n_{ij} - 2n_{ik}n_{jk}),
$$
which proves~\eqref{eq:Uijknorm2} and thus that 
$U_{ij,k}\in\mathcal{C}^+$.
\end{proof}

As we described in Section~\ref{sec:Conceptualframework}, and used to
determine the form of the $U_{ij,k}$, we will want the current to be 
conserved at interior elements, so it is useful to record some 
identities for sums of local currents.  For 
$S\subseteq [n]$, let $S^{\text{c}} = [n]\setminus S$ and define
$$
v_S = \sum_{i\in S} v_i
 \quad\text{and}\quad
m_S = \sum_{\{i,j\}\subseteq S} m_{ij}.
$$

\begin{lemma}\label{lem:DTsums}
Let $n\ge2$ and let $v_1,\ldots,v_n\in\RR^{d-1}$ be unit vectors with
$\sum_{i=1}^n v_i = 0$.  For every pair $D = \{i,j\}\subseteq[n]$,
\begin{equation}\label{eq:sumkUijk}
  \sum_{k\in[n]\setminus D} U_{ij,k}
  =
  \bigl(2n + (n-4)m_{ij},(n + m_{ij})(v_i + v_j)\bigr)
  =
  (n + m_{ij})(V_i + V_j) + (n-6)m_{ij}(1,0).
\end{equation}
For every triple $T = \{i,j,k\}\subseteq[n]$,
\begin{equation}\label{eq:3Usum}
  U_{ij,k} + U_{ik,j} + U_{jk,i} = 3 m_T(1,0).
\end{equation}
For every subset $S\subseteq[n]$,
\begin{equation}\label{eq:sumofm}
  m_S = \frac12(|S|^2 - \|v_S\|^2),
\end{equation}
and
\begin{equation}\label{eq:diffofm}
  m_S - m_{S^{\text{c}}} = \frac12 n(2|S| - n).
\end{equation}
\end{lemma}

\begin{proof}
Each of these identities is a straightforward calculation.  Setting
$\alpha = \beta = 1$ in~\eqref{eq:Uijksumalphabeta} 
gives~\eqref{eq:sumkUijk}.

Noticing that the time component of $U_{ij,k}$ is invariant under 
permutations of $T$ while the spatial components cancel in this sum
proves~\eqref{eq:3Usum}.

Finally,
$$
m_S = \sum_{\{i,j\}\subseteq S}(1 - v_i\cdot v_j)
    = \frac12|S|(|S|-1) - \frac12\|v_S\|^2 + \frac12|S|,
$$
which proves~\eqref{eq:sumofm}, and then
$$
m_S - m_{S^{\text{c}}} 
 =
\frac12\bigl(|S|^2 - \|v_S\|^2\bigr) - 
\frac12\bigl((n-|S|)^2 - \|v_{S^{\text{c}}}\|^2\bigr),
$$
which becomes~\eqref{eq:diffofm} because $\sum_i v_i = 0$ implies
$\|v_S\|^2 = \|v_{S^{\text{c}}}\|^2$.
\end{proof}

Not surprisingly, since it motivated the choice of $U_{ij,k}$, $n = 6$ 
is a special value for these identities.  For $\mathcal{B}_6$, the 
last term in~\eqref{eq:sumkUijk} vanishes; also the 
identity~\eqref{eq:diffofm} implies that $m_S = m_{S^{\text{c}}}$ for 
all triples $S$.

\section{$\mathcal{B}_6$ is not a sphere order}
\label{sec:B6isnotasphereorder}

We are now ready to state and prove the main theorem.

\begin{theorem}\label{thm:maintheorem}
The Boolean lattice $\mathcal{B}_6$ is not a sphere order in any 
dimension.
\end{theorem}

\begin{proof}
Following the discussion at the beginning of 
Section~\ref{sec:EmbeddingconstraintsforBn} and then
Lemma~\ref{lem:witnesses}, if $\mathcal{B}_6$ has a sphere 
representation, each of the cover differences lies in $\mathcal{C}^+$ 
and there are $V_i = (1,v_i)$ for $i\in[6]$ with 
$\sum_{i\in[6]} v_i = 0$ such that for all $D\subset[6]$ with 
$|D| = 2$ and $i\in D$,
\begin{equation}\label{eq:witness}
  \langle V_i,X_{D^{\text{c}}} - X_D\rangle_L < 0.
\end{equation}
With these witnesses fixed, it suffices to consider only the ``middle'' 
subposet $M\subset\mathcal{B}_6$ consisting of the 2, 3, and 4 element 
subsets of $[6]$, so using these $v_i$, define $U_{ij,k}$ according 
to~\eqref{eq:Uijk}, and for each $D = \{i,j\}\subset\{i,j,k\} = T$, 
define local currents
\begin{equation}\label{eq:localcurrentsdef}
  F_{DT} = U_{ij,k} = F_{T^{\text{c}}D^{\text{c}}}.
\end{equation}
These are in $\mathcal{C}^+$ by Lemma~\ref{lem:triples}, so for each 
of the cover relations $D\link T$ and $T^{\text{c}}\link D^{\text{c}}$ 
the current has nonnegative Lorentz inner product with the difference 
vector, be it $X_T - X_D$ or $X_{D^{\text{c}}} - X_{T^{\text{c}}}$.

Furthermore, this current is conserved:  For $T = \{i,j,k\}$,
$$
(\partial F)_T 
 =
\sum_{D\subset T}F_{DT} - \sum_{T\subset Q}F_{TQ},
$$
where $|D| = 2$ and $|Q| = 4$.  By~\eqref{eq:3Usum}, the incoming 
current is $3m_T(1,0)$.  The complementarity of the local currents 
defined in~\eqref{eq:localcurrentsdef} means that the outgoing
current is the same as incoming current to $T^{\text{c}}$, which 
by~\eqref{eq:diffofm}, is also $3m_T(1,0)$.  

Thus, since the local currents $F_{AB}$ and the cover differences are 
all in $\mathcal{C}^+$, we have
\begin{align}
  0 &\le \sum_{A\link B\in M}\langle F_{AB},X_B - X_A\rangle_L\notag\\
    &=   \sum_{\substack{D\in M\\ |D|=2}} 
           \langle(\partial F)_D,X_D\rangle_L +
         \sum_{\substack{Q\in M\\ |Q|=4}} 
           \langle(\partial F)_Q,X_Q\rangle_L                 \notag\\
    &=  -\sum_{\{i,j\}\subset[6]} 
         (6 + m_{ij})\langle V_i + V_j,X_{\{i,j\}}\rangle_L 
        +\sum_{\{i,j\}\subset[6]}
         (6 + m_{ij})\langle V_i + V_j,X_{\{i,j\}^{\text{c}}}\rangle_L,
                                                \label{eq:nonnegative}
\end{align}
where the first equality uses Stokes' 
theorem~(\ref{eq:Stokestheorem},\ref{eq:boundaryterms}) and the 
conservation of the current, and the second equality follows 
from~\eqref{eq:sumkUijk} and the complementarity of the 
current~\eqref{eq:localcurrentsdef}.  Rearranging the sums 
in~\eqref{eq:nonnegative}, we have
$$
0 
 \le 
\sum_{D = \{i,j\}\subset[6]}
(6 + m_{ij})\bigl(\langle V_i,X_{D^{\text{c}}} - X_D\rangle_L +
                  \langle V_j,X_{D^{\text{c}}} - X_D\rangle_L
            \bigr).
$$
But each term in the sum is negative, because $6+m_{ij} > 0$ from 
the definition~\eqref{eq:nmdef} and $V_i$ and $V_j$ witness the 
non-relation between $D = \{i,j\}$ and 
$D^{\text{c}}$~\eqref{eq:witness}.  This contradicts the existence of
an embedding of $\mathcal{B}_6$ in Minkowski space, and thus also any 
realization as a sphere order.
\end{proof}

In fact, the current in the proof extends to all five nontrivial 
layers, giving the certificate~\eqref{eq:certificate} that motivated 
its construction.  For each pair $D=\{i,j\}$, extend $F$ by setting
$$
F_{A_iD} = F_{D^{\text{c}}B_i} = (6 + m_{ij})V_i,
$$
and similarly with $i$ and $j$ exchanged.  These additional local 
currents are future-directed and null.  At $D$, the two incoming 
currents sum to $(6 + m_{ij})(V_i + V_j)$, which equals the outgoing 
current by~\eqref{eq:sumkUijk}.  Complementation gives conservation at 
$D^{\text{c}}$ as well.  Together with conservation at triples, this 
shows that the extended current is conserved at every element of ranks 
$2$, $3$, and $4$.

Since 
$$
\sum_{j\ne i} m_{ij} = 5 - v_i\cdot\sum_{j\ne i} v_j = 6,
$$
the total current leaving $A_i$, and entering $B_i$, is
$$
\sum_{j\ne i}(6 + m_{ij})V_i = 36V_i.
$$
Thus the certificate~\eqref{eq:certificate} becomes
$$
\partial F = 36\sum_{i=1}^6 V_i\bigl([B_i] - [A_i]\bigr).
$$
For a putative embedding, Stokes' theorem therefore gives the
contradiction
$$
0
 \le
\sum_{\substack{A\link B\\A,B\in\mathcal{B}_6^{\circ}}}
  \langle F_{AB},X_B - X_A\rangle_L
 =
36\sum_{i=1}^6\langle V_i,X_{B_i} - X_{A_i}\rangle_L
 <
0.
$$

Restricting the certificate to the middle three layers eliminates
the outer null currents while retaining their fluxes as boundary
data.  Also, since the constructions of the witnesses in 
Lemma~\ref{lem:witnesses} and this extended current use neither 
$\emptyset$ nor $[6]$, the same argument proves that the 62-element 
poset $\mathcal{B}_6^{\circ}$ is not a sphere order.

\section{Spherical configurations}
\label{sec:sphericalconfigurations}

El-Zahar and Fateen~\cite{ElZaharFateen1998} conjectured the 
nonexistence of a configuration of $2^6-2$ points on $\SSS^4$ 
satisfying antipodality conditions and a weak upper bound of $\pi$ on 
the length of every maximal chain.  They observed that this 
nonexistence would imply that $\mathcal{B}_6$ is not a sphere order.  
The same Minkowski space geometry and conserved current we have used 
to prove Theorem~\ref{thm:maintheorem} also establish their 
conjecture, in fact on a unit sphere of any finite dimension:

\begin{proposition}\label{prop:ElZaharFateen}
Let $d\ge 3$, and assign distinct points $K(A)$ of the unit sphere in
$\RR^{d-1}$ to the elements $A\in\mathcal{B}_6^\circ$.  Suppose that
$$
K(B_i) = -K(A_i),
$$
for $i\in[6]$, again using the notation of Lemma~\ref{lem:witnesses}.
Write $\phi(A,B) = \arccos\bigl(K(A)\cdot K(B)\bigr)$ for the
spherical distance between $K(A)$ and $K(B)$.  Then there is a maximal 
chain
$$
C_1\link C_2\link C_3\link C_4\link C_5
\in\mathcal{B}_6^{\circ},
$$
with $|C_r| = r$, such that
$$
\sum_{r=1}^4\phi(C_r,C_{r+1}) > \pi.
$$
\end{proposition}

\begin{proof}
Suppose instead that every maximal chain has total spherical length
at most $\pi$.  For $A\in\mathcal{B}_6^\circ$, let $h_A$ be the
maximum of the sums of spherical distances along saturated chains
from a singleton to $A$, with $h_{A_i} = 0$.  Then
\begin{equation}\label{eq:EZFheights}
  A\link B
  \quad\Longrightarrow\quad
  h_B\ge h_A + \phi(A,B).
\end{equation}
Let $M$ again denote the subposet consisting of the subsets of
cardinalities $2$, $3$, and $4$.  Since the points $K(A)$ are
distinct, every cover has positive spherical length, so $h_A > 0$ for
$A\in M$.  Extending a chain attaining $h_A$ to a co-atom adds a
positive length; hence $h_A < \pi$ for $A\in M$.

Let $p_i = K(A_i)$.  If $i\in A$, a saturated chain from $A_i$ to $A$
and the triangle inequality give $\phi(A_i,A)\le h_A$.  If
$i\notin A$, extend a chain attaining $h_A$ to $B_i$.  The bound on
maximal chain lengths and the triangle inequality give
$h_A+\phi(A,B_i)\le\pi$.  Antipodality implies 
$\phi(A,B_i) = \pi - \phi(A_i,A)$, so in this case 
$h_A\le \phi(A_i,A)$.  Taking cosines, for every $A\in M$ we have
\begin{equation}\label{eq:EZFpointsigns}
  p_i\cdot K(A)
    \begin{cases}
      \ge\cos h_A & \text{if\ }i\in A;                             \\
      \le\cos h_A & \text{if\ }i\notin A.
    \end{cases}
\end{equation}

Using the image of $\mathcal{B}_6^{\circ}$ in the sphere we can 
construct a map into Minkowski space:  for $A\in M$, define
\begin{equation}\label{eq:EZFlift}
  X_A = \left(-\cot h_A,\frac{K(A)}{\sin h_A}\right)\in\MM^d.
\end{equation}
Formula~\eqref{eq:EZFlift} uses the global conformal parametrization 
of the unit de Sitter hyperboloid $\langle X,X\rangle_L = -1$, with 
induced metric $(\dd h^2 - \dd\Omega^2)/\sin^2 h$, where $\dd\Omega^2$ 
is the round metric on $\SSS^{d-2}$.  The full conformal-time interval 
is $0 < h < \pi$, and the height inequalities~\eqref{eq:EZFheights} 
are precisely the causal-order inequalities in these coordinates.  If 
$A\link B\in M$, we can compute
$$
\langle X_B - X_A,X_B - X_A\rangle_L
 =
\frac{2\bigl(\cos\phi(A,B)-\cos(h_B-h_A)\bigr)}
     {\sin h_A\sin h_B}
\ge0,
$$
where the inequality follows from~\eqref{eq:EZFheights} and
$0 < h_A < h_B < \pi$.  Since $-\cot h$ is strictly increasing on
$(0,\pi)$, the time component of $X_B - X_A$ is positive.  Thus every
cover difference is a nonzero vector in $\mathcal{C}^+$.

The six vectors $N_i = (1,-p_i)$ are distinct future-directed null
vectors.  By~\eqref{eq:EZFpointsigns},
$$
\langle N_i,X_A\rangle_L
  =
\frac{p_i\cdot K(A)-\cos h_A}{\sin h_A}
\begin{cases}
  \ge0 &\text{if\ }i\in A;                                          \\
  \le0 &\text{if\ }i\notin A.
\end{cases}
$$
Apply Lemma~\ref{lem:center} to the $N_i$, and apply the same Lorentz
transformation to all the $X_A$.  Write the normalized null vectors
as $V_i = (1,v_i)$, with $\sum_i v_i = 0$, and retain the notation 
$X_A$ for the transformed vectors.  The $v_i$ remain distinct, the 
cover differences remain nonzero and future-directed causal, and the
positive rescalings of the witnesses preserve the preceding signs.
In particular, for every pair $D\subset[6]$ and every $i\in D$,
\begin{equation}\label{eq:EZFweakwitness}
  \langle V_i,X_{D^{\text{c}}} - X_D\rangle_L\le0.
\end{equation}

Now again define $F$ by~\eqref{eq:localcurrentsdef}.  As in the proof 
of Theorem~\ref{thm:maintheorem}, its local currents lie in
$\mathcal{C}^+$ and~\eqref{eq:nonnegative} holds.  To obtain a 
contradiction from the weak inequalities~\eqref{eq:EZFweakwitness}, we 
show that at least one local current is timelike.  Indeed, choose 
$i\ne j$ with $v_i + v_j\ne 0$;  such a pair exists because the six 
unit vectors are distinct.  Since $m_{ij}>0$, the 
identity~\eqref{eq:Uijknorm2} shows that $U_{ij,k}$ can be null only 
if $v_k$ is parallel to $v_i + v_j$.  At most two of the four distinct 
unit vectors with $k\notin\{i,j\}$ can satisfy this condition.  Thus 
some $U_{ij,k}$ is timelike.

By Lemma~\ref{lem:cone}, at least one term in the Stokes sum
is strictly positive, and all the others are nonnegative.  Stokes' 
theorem, the boundary calculation in~\eqref{eq:nonnegative}, and 
\eqref{eq:EZFweakwitness} now give
\begin{align*}
0
 <
\sum_{\substack{A\link B\\ A,B\in M}}
  \langle F_{AB},X_B - X_A\rangle_L
 =
\sum_{D=\{i,j\}\subseteq[6]}(6+m_{ij})
  \langle V_i + V_j,X_{D^{\text{c}}} - X_D\rangle_L
\le 0,
\end{align*}
which is a contradiction.
\end{proof}

\section{Discussion}
\label{sec:discussion}

Since $\mathcal{B}_6$ is a subposet of $\mathcal{B}_n$ for $n\ge 6$, 
none of the latter are sphere orders.  $\mathcal{B}_5$, however, is a 
sphere order~\cite{Meyer1989,BrightwellWinkler1989,ElZaharFateen1998}, 
so we have found the smallest Boolean lattice that is not a sphere 
order.  Being the product of $6$ nontrivial chains $\mathcal{B}_6$ has 
order dimension $6$~\cite{KellyTrotter1982}, while Felsner, Fishburn 
and Trotter's theorem asserts that there is a finite poset with order 
dimension $3$ that is not a sphere order~\cite{FFT1999}.  Observing 
the cardinality of $\mathcal{B}_6$, knowing that $[4]^3$ is a circle 
order~\cite{ElZaharFateen1998}, and having tried and failed to find 
numerical embeddings for $[5]^3$ and $[3]^4$, we conjecture that  
neither is a sphere order.


\begin{thebibliography}{99}

\bibitem{BLMS1987}
Luca Bombelli, Joohan Lee, David A. Meyer and Rafael D. Sorkin,
``Space-time as a causal set'',
\textsl{Physical Review Letters} \textbf{59} (1987) 521--524.

\bibitem{Meyer1989}
David A. Meyer,
{\sl The Dimension of Causal Sets},
Ph.D.\ thesis (MIT, Mathematics 1989).

\bibitem{BrightwellWinkler1989}
Graham Brightwell and Peter Winkler,
``Sphere orders'',
\textsl{Order} \textbf{6} (1989) 235--240.

\bibitem{Meyer1993}
David A. Meyer,
``Spherical containment and the Minkowski dimension of partial 
  orders'',
\textsl{Order} \textbf{10} (1993) 227--237.
  
\bibitem{FFT1999}
Stefan Felsner, Peter Fishburn and William T. Trotter,
``Finite three dimensional partial orders which are not sphere 
  orders'',
\textsl{Discrete Mathematics} \textbf{201} (1999) 101--132.

\bibitem{Daughton1998}
Alan Daughton,
``An investigation of the symmetric case of when causal sets can embed 
  into manifolds'',
\textsl{Classical and Quantum Gravity} \textbf{15} (1998) 3427--3434.

\bibitem{ElZaharFateen1998}
Mohamed H. El-Zahar and Lobna A. Fateen,
``On sphere orders'',
\textsl{Discrete Mathematics} \textbf{185} (1998) 249--253. 

\bibitem{BoydVandenberghe2004}
Stephen Boyd and Lieven Vandenberghe,
\textsl{Convex Optimization}
(Cambridge:  Cambridge University Press 2004).

\bibitem{DHLM2005}
Mathieu Desbrun, Anil N. Hirani, Melvin Leok and Jerrold E. Marsden,
``Discrete exterior calculus'' (2005),
\href{https://arxiv.org/abs/math/0508341}{\texttt{arXiv:math/0508341}}.

\bibitem{Alexander1935}
James W. Alexander,
``On the chains of a complex and their duals'',
\textsl{Proceedings of the National Academy of Sciences of the
  United States of America} \textbf{21} (1935) 509--511.

\bibitem{deRham1931}
Georges de Rham,
``\textit{Sur l'analysis situs des vari\'et\'es \`a $n$ dimensions}'',
\textsl{Journal de Math\'ematiques Pures et Appliqu\'ees},
  \textit{neuvi\`eme s\'erie} \textbf{10} (1931) 115--200.
  
\bibitem{DouadyEarle1986}
Adrien Douady and Clifford J. Earle,
``Conformally natural extension of homeomorphisms of the circle'', 
\textsl{Acta Mathematica} \textbf{157} (1986) 23--48.

\bibitem{CantarellaSchumacher2022}
Jason Cantarella and Henrik Schumacher,
``Computing the conformal barycenter'',
\textsl{SIAM Journal on Applied Algebra and Geometry} \textbf{6} 
  (2022) 503--530.
  
\bibitem{KapovichMillson1996}
Michael Kapovich and John J. Millson,
``The symplectic geometry of polygons in Euclidean space'',
\textsl{Journal of Differential Geometry} \textbf{44} (1996) 479--513.

\bibitem{KellyTrotter1982}
David Kelly and William T. Trotter, Jr.,
``Dimension theory for ordered sets'' 
in Ivan Rival, ed., 
\textsl{Ordered Sets}, NATO ASI Series C \textbf{83} 
(Dordrecht:  D. Reidel 1982) 171--211.

\end{thebibliography}
\end{document}